\documentclass{article}
\usepackage{graphicx}
\usepackage{epsfig}
\usepackage{amsmath}
\usepackage{amssymb}
\usepackage{amsthm}
\usepackage{stmaryrd}
\usepackage{url}
\usepackage{longtable}
\usepackage[figuresright]{rotating}
\usepackage{color}
\usepackage{amsfonts}
\usepackage{amsthm}
\usepackage{mathrsfs}
\usepackage{stmaryrd}
\usepackage{wasysym}
\usepackage{xypic}
\xyoption{curve}
\usepackage{cases}
\usepackage{hyperref}

\newcommand{\Field}[1]{#1}			
\newcommand{\SL}[2]{\mathrm{SL}_{#1}(#2)}    
\newcommand{\GL}[2]{\mathrm{GL}_{#1}(#2)}		
\newcommand{\Q}{\mathbb{Q}}            
\newcommand{\Z}{\mathbb{Z}}            
\newcommand{\conj}[2]{{#1}^{#2}}  		

\newcommand{\pIdeal}[1]{\mathfrak{p}_{#1}}
\newcommand{\pElem}{\varpi}

\newcommand{\oRing}[1]{\mathfrak{o}_{#1}}
\newcommand{\val}{\mathrm{v}} 		

\newcommand{\ramInd}[1]{\mathrm{e}_{#1}}	 

\newcommand{\grp}[1]{\mathrm{#1}}		
\newcommand{\G}{\mathrm{G}}							
\newcommand{\Mat}[1]{\mathtt{#1}}		
\newcommand{\trace}[1]{\mathtt{tr}\left(#1\right)}

\newcommand{\vertex}[3]{\begin{bmatrix}#1\\#2\end{bmatrix}_{\displaystyle#3}}		

\newcommand{\BLPair}[2]{\left\langle#1,#2\right\rangle}

\newcommand{\quotChar}[1]{\chi_{_{\displaystyle #1}}} 

\newcommand{\nin}{\not{\!\!\in\hspace{4pt}}}

\theoremstyle{plain}
\newtheorem{theorem}{Theorem}
\newtheorem*{theorem*}{Theorem}
\newtheorem{lemma}{Lemma}
\newtheorem*{lemma*}{Lemma}

\theoremstyle{definition}

\newtheorem*{definition*}{Definition}
\newtheorem*{notation*}{Notation}
\theoremstyle{comment}

\newtheoremstyle{named}{}{}{\itshape}{}{\bfseries}{.}{.5em}{#1 #3}
\theoremstyle{named}
\newtheorem*{namedtheorem}{Theorem}

\title{Moy-Prasad maps for $\mathrm{SL}(2)$ over extensions of $\Q_2$}
\author{Terence Joseph K\i vran-Swaine\\
The New York City College of Technology\\
The City University of New York\\
New York, NY 11201\\
\texttt{tkivran-swaine@citytech.cuny.edu}}
\date{9/27/2011}
\begin{document}
\maketitle
\begin{abstract}A parametrization for characters of abelian quotients of compact subgroups of $\grp{G}:=\SL{2}{\Field{F}}$ is constructed for $\Field{F}$ an algebraic extension of $\Q_2$ that corresponds to the Moy-Prasad maps of fields of odd residual characteristic.\end{abstract}
In his Ph.D. thesis, \cite{Sh1966}, Joseph Shalika established that the representations of Weil (\cite{W1964}) form an exhaustive list of irreducible, cuspidal representations over $\grp{G}:=\SL{2}{\Field{F}}$ for $\Field{F}$ a local field of odd residual characteristic.  To that end, Shalika employs a now classical filtration of the compact subgroups of $\grp{G}$ and then lists the possible characters on compact subgroups of $\grp{G}$ contained in such representations to count the representations of $\grp{G}$ of a particular ``level''.\\
In his 1972 paper, \cite{C1972}, Casselman extended this technique to construct the irreducible cuspidal representations which occur in the construction of Weil for $\SL{2}{\Field{F}}$ with $\Field{F}$ of even residual characteristic.  In 1976, Nobs establish that exactly four ``exceptional" representations for $\SL{2}{\Q_2}$ existed outside the construction of Weil.  Later that year, Nobs along with Wolfart identified these representations in \cite{NW1976} by using tensor products of other representations.\\
In his 1978 papers, \cite{K1978p1} and \cite{K1978p2}, Kutzko built on the parametrization of Shalika to explicitly construct cuspidal representations of $\GL{2}{\Field{F}}$ for arbitrary residual characteristic.  An alternate presentation of this approach can be found in Kutzko's 1972 Ph.D. Thesis, \cite{K1972}.  His parametrization became know as the theory of cuspidal types. \\
This technique was modified back to $\SL{2}{\Field{F}}$ of odd residual characteristic by Mandersheid in his 1984 papers, \cite{M1} and \cite{M2}.  In 1994 Moy and Prasad proved that a more general approach for classifying irreducible representations could be employed for arbitrary reductive groups over arbitrary $p$.  While their approach differed from that of Kutzko, they rely on the implicit existence of a parametrization of characters corresponding to that used in Shalika's thesis. (\cite{MP1994}).\\
The functions underlying the parametrization of characters employed by Moy and Prasad became known as Moy-Prasad maps.  They provide an $\grp{G}$-isomorphism to certain characters of compact subgroups of $\grp{G}$ from quotients of fractional ideals in the Lie Algebra corresponding to $\grp{G}$.  Resently they have been employed in work by Yu, \cite{Y2001} and in the exposition the harmonic analyis of $\SL{2}{\Field{F}}$ for $p\neq 2$ by Adler, Debacker, Sally and Spice, \cite{ADSS2011}.\\
As presented in the literature, these $\G$-isomorphisms factors through a $p$-adic Killing form on the Lie Algebra in question.  As the killing form on the trace zero matrices $\Z/2\Z$ is known to be degenerate, the traditional formulation of Moy-Prasad maps does not freely trasfer to $\SL{2}{\Field{F}}$ for $\Field{F}$ of residual characteristic two.  In this note I develop a counterpart for Moy-Prasad maps for the case where $p=2$.\\
This work is a generalization of a portion of my dissertation, \cite{K-S2011} and is the first in a sequence of papers which will develop the representation theory and harmonic analysis of $\SL{2}{\Field{F}}$ where $p=2$, employing the methods of Kutzko and Mandersheid.
\begin{notation*}
I will use $l$, $m$ and $n$ to denote integers.  Here is some notation I will employ:
\begin{itemize}
\item Let $\Field{F}$ be an algebraic extension of $\Q_2$ of ramification index $\ramInd{}$.
\item Let $\G=\SL{2}{\Field{F}}$.
\item Let $\oRing{}$ denote the ring of integers of $\Field{F}$.  
\item Let $\pIdeal{}$ be the prime ideal of $\oRing{}$.
\item Let $\pElem$ be a local uniformizing parameter that generates $\pIdeal{}$.
\item Denote the valuation of $x$ in $\Field{F}^\times$ by $\val(x)$.
\item Let $\grp{K}$ denote a maximal compact subgroup of $\grp{G}$, namely $\SL{2}{\oRing{}}$.
\item For $n,m\geq 0$, let $\grp{K}_n^m$ denote the compact subgroup of $\grp{G}$ which consists of elements of the form,
\begin{equation}
1_2+\begin{pmatrix}\pIdeal{}^n&\pIdeal{}^{n+m}\\
\pIdeal{}^{n+m}&\pIdeal{}^n
\end{pmatrix}.
\end{equation}
\end{itemize} 

\end{notation*}
\noindent In this paper I offer the following results.
\begin{theorem}\label{thm:Nrml}
If $\ramInd{},n\geq m$, then $\grp{K}_{n}^{m}$
is a normal subgroup of $\grp{K}$.
\end{theorem}
\begin{theorem}\label{thm:dual}
If $m\leq\ramInd{}$, then the set of characters of the quotient
$\grp{K}_{n}^{m}/\grp{K}_{2n}^{m}$
as a $\grp{G}$-set is parametrized by matrices of the form
\begin{equation*}
\begin{pmatrix}\pIdeal{}^{-2n-\ramInd{}}/\pIdeal{}^{-n-\ramInd{}}
&\pIdeal{}^{-2n-m}/\pIdeal{}^{-n-m}\\
\pIdeal{}^{-2n-m}/\pIdeal{}^{-n-m}&\pIdeal{}^{-2n-\ramInd{}}/\pIdeal{}^{-n-\ramInd{}}
\end{pmatrix}.
\end{equation*}
\end{theorem}
\subsection*{Proof of Theorem \ref{thm:Nrml}}
\begin{notation*}
In both this proof and the next I will denote $u=\dfrac{\pElem^{\ramInd{}}}2$.  Note that $u\in\oRing{}^\times$. \\
I recall two famous subgroups of $\grp{K}$.
\begin{itemize}
\item For $n\geq 1$, let $\grp{K}_{n}$ be the kernel in $\grp{K}$ of the map induced by the natural homomorphism $\oRing{}\rightarrow\oRing{}/\pIdeal{}^n$ and let $\grp{K}_0=\grp{K}$.
\item For $n\geq 1$, let $\grp{B}_n$ be the preimage in $\grp{K}$ of the group of upper triangular matrices of the map induced by the natural map $\oRing{}\rightarrow\oRing{}/\pIdeal{}^n$  and let $\grp{B}_0=\grp{K}$.
\end{itemize}
For $n\geq 0 $, I will also denote by $\vertex{x}{y}{n}$, the homothety class of $(x,y)\in\oRing{}\times\oRing{}$ in the projective line over $\oRing{}/\pIdeal{}^{n}$.  For my purposes, the projective line over $\oRing{}/\pIdeal{}^{0}$ is a singleton set. \\
\end{notation*}
Note that for the calculations that follow, one may consider $\left[\begin{smallmatrix}x\\y\end{smallmatrix}\right]_{n}$
 as a vertex on the Bruhat-Tits tree that is distance $n$ from the vertex stabilized by $\grp{K}$ as in \cite{S1977}.  While this perspective can be  insightful with respect to Theorem \ref{thm:Nrml}, it is not necessary for purposes of proof. \\
I consider the transitive action of $\grp{K}$ on the projective line over $\oRing{}/\pIdeal{}^{n}$: 
\begin{equation*}
\begin{pmatrix}
a&b\\c&d
\end{pmatrix}\vertex{x}{y}{n}=\vertex{ax+by}{cx+dy}{n}
\end{equation*}
It is well known and easily verified that the stabilizer of the projective point $\left[\begin{smallmatrix}1\\0\end{smallmatrix}\right]_{n}$ is none other than $\grp{B}_{n}$.  Consequently, for $\gamma\in\grp{K}$ the conjugate $\conj{\grp{B}_n}\gamma$ is the stabilizer of $\gamma^{-1} \left[\begin{smallmatrix}1\\0\end{smallmatrix}\right]_{n}$.\\
I now consider the group of matrices: 
\begin{equation*}
\grp{B}_{n+m}\cap\grp{K}_{n}=1_2+\begin{pmatrix}\pIdeal{}^n&\pIdeal{}^{n}\\
\pIdeal{}^{n+m}&\pIdeal{}^n
\end{pmatrix}
\end{equation*}
\begin{lemma}
If $n,\ramInd{}\geq m$, the group $\grp{K}_{n}^{m}$
is the intersection of each of the conjugates of $\grp{B}_{n+m}\cap\grp{K}_{n}$ under $\grp{K}$.
\end{lemma} 
\begin{proof}
First it is easy to observe that $\grp{K}_{n}^{m}$ is the intersection of $\grp{B}_{n+m}\cap\grp{K}_{n}$ with the stabilizer of $\left[\begin{smallmatrix}0\\1\end{smallmatrix}\right]_{n+m}$, since the latter is precisely the group of determinant-one matrices of the form
\begin{equation*}
\begin{pmatrix}
\oRing{}^\times&\pIdeal{}^{m+n}\\
\oRing{}&\oRing{}^\times
\end{pmatrix}.
\end{equation*}
To show that is a subset of the other conjugates of $\grp{B}_{n+m}\cap\grp{K}_{n}$ and hence the intersection of all conjugates of $\grp{B}_{n+m}\cap\grp{K}_{n}$, I compute the the action of $\grp{K}_{n}^{m}$ on an arbitrary point of the projective line, illustrating that it is trivial.
\begin{align*}
\begin{pmatrix}
1+\pElem^n a& \pElem^{n+m}b\\
 \pElem^{n+m}c & 1+\pElem^n d\\
\end{pmatrix}&\vertex{x}{1}{n+m}=
\begin{pmatrix}1+\pElem^n a& 0\\
0 & (1+\pElem^n a)^{-1}\\
\end{pmatrix}\vertex{x}{1}{n+m}\\
&= \vertex{(1+\pElem^n a)x}{(1+\pElem^n a)^{-1}1}{n+m}\\
&= \vertex{(1+\pElem^n a)^2x}{1}{n+m}\\
&= \vertex{(1+2\pElem^na+\pElem^{2n}a^2)x}{1}{n+m}\\
&= \vertex{(1+\pElem^{n+\ramInd{}}\dfrac{a}{u}+\pElem^{2n}a^2)x}{1}{n+m}\\
&= \vertex{x}{1}{n+m},
\end{align*}
because $\ramInd{}, n \geq m$. 
\end{proof}
Theorem \ref{thm:Nrml} is a direct consequence of the preceding lemma.
\subsection*{Proof of Theorem \ref{thm:dual}}
The parametrization of characters referenced implicitly in Theorem \ref{thm:dual} is defined as follows.
\begin{equation}
\quotChar{\Mat{A}}(X):=\chi\bigl(\trace{(X-1)\Mat{A}} \bigr).
\end{equation}
The proof of Theorem \ref{thm:dual} is a consequence of the following three lemmas.
\begin{lemma}\label{lem:pair}
The pairing $\BLPair{\Mat{A}}{\Mat{B}}=\trace{\Mat{A}\Mat{B}}$ of trace-zero matrices,
\begin{equation}
\begin{pmatrix}
\pIdeal{}^{n}&\pIdeal{}^{n+m}\\
\pIdeal{}^{n+l}&\pIdeal{}^{n}\\
\end{pmatrix}
\times\begin{pmatrix}
\pIdeal{}^{-n-\ramInd{}}&\pIdeal{}^{-n-l}\\
\pIdeal{}^{-n-m}&\pIdeal{}^{-n-\ramInd{}}\\
\end{pmatrix}
\rightarrow
\oRing{},
\end{equation}
is bilinear and non-degenerate in the sense that if $\BLPair{\Mat{B}}{\Mat{A}}\in\pIdeal{}$ for every trace-zero matrix, $\Mat{B}\in\left( \begin{smallmatrix}
\pIdeal{}^{n}&\pIdeal{}^{n+m}\\
\pIdeal{}^{n+l}&\pIdeal{}^{n}\\
\end{smallmatrix}\right)$, then 
\begin{equation*}
\Mat{A}\in\pIdeal{}\cdot\begin{pmatrix}
\pIdeal{}^{-n-\ramInd{}}&\pIdeal{}^{-n-l}\\
\pIdeal{}^{-n-m}&\pIdeal{}^{-n-\ramInd{}}\\
\end{pmatrix}
\end{equation*}
\end{lemma}
\begin{proof}
This follows from direct calculation.
\begin{equation*}
\begin{split}
\trace{\pElem^n\begin{pmatrix}
a_1&\pElem^m a_2\\
\pElem^l a_3& -a_1
\end{pmatrix}
\pElem^{-n}\begin{pmatrix}
\pElem^{-\ramInd{}}b_1&\pElem^{-l}b_3\\
\pElem^{-m}b_2&-\pElem^{-\ramInd{}}b_1
\end{pmatrix}}\\
\quad=\trace{\begin{pmatrix}
a_1&\pElem^m a_2\\
\pElem^l a_3& -a_1
\end{pmatrix}
\begin{pmatrix}
\frac1{2u}b_1&\pElem^{-l}b_3\\
\pElem^{-m}b_2&-\frac1{2u}b_1
\end{pmatrix}}\\
=\frac1{2u}a_1b_1+a_2b_2+a_3b_3+\frac1{2u}a_1b_1\\
=\frac1{u}a_1b_1+a_2b_2+a_3b_3
\end{split}
\end{equation*}
If $\val({b_i})=0$ setting $a_i=1$ and $a_j=0$ for $j\neq i$ ensures that $\BLPair{\Mat{B}}{\Mat{A}}\nin\pIdeal{}$.\\
This is a bilinear pairing since $\trace{x\Mat{B}\Mat{A}}+\trace{y\Mat{C}\Mat{A}}=\trace{(x\Mat{B}+y\Mat{C})\Mat{A}}$ and similarly in the right-hand arguement. \end{proof}
\begin{lemma}\footnote[1]{Note that the second and third lemmas hold even when $\Field{F}$ is an algebraic extension of $\Q_p$ with $p\neq 2$.} \label{lem:Gmap}
The map $\Mat{A}\mapsto\quotChar{\Mat{A}}$ is a group homomorphism and a map of $\grp{G}$-sets.  That is,
\begin{equation*}
\quotChar{\Mat{A}}^g(X):=\quotChar{\Mat{A}}(X^g)=\quotChar{\Mat{A}^g}(X).
\end{equation*}
\end{lemma}
\begin{proof}
First note that by Lemma \ref{lem:pair}, 
\begin{align*}
\quotChar{\Mat{A}}(X)&\cdot\quotChar{\Mat{B}}(X)=
\chi\left(\BLPair{X}{\Mat{A}}\right)\chi\left(\BLPair{X}{\Mat{B}}\right)\\
&=\chi\left(\BLPair{X}{\Mat{A}}+\BLPair{X}{\Mat{B}}\right)\\
&=\chi\left(\BLPair{X}{\Mat{A}+\Mat{B}}\right)=\quotChar{\Mat{A}+\Mat{B}}(X).
\end{align*}
Employing the fact that the trace of a matrix is invariant under conjugation, I calculate that
\begin{align*}
\quotChar{\Mat{A}}(X^g)&=\chi\left(\BLPair{X^g-1}{\Mat{A}}\right)\\
&=\chi\left(\BLPair{(X-1)^g}{\Mat{A}}\right)\\
&=\chi\left(\trace{(X-1)^g\Mat{A}}\right)\\
&=\chi\left(\trace{(X-1)\Mat{A}^g}\right)\\
&=\quotChar{\Mat{A}^g}(X),
\end{align*}
which illustrates the preservation of the $\G$-action.\\
\end{proof}
\begin{lemma}\footnotemark\label{lem:Xminus1}
If $n\geq 1$ and $l, m\geq -1$, the map $X\mapsto X-1$ induces an isomorphism from a multiplicative quotient of determinate-one matrices to an additive quotient of trace-zero matrices:
\begin{equation*}
\bigl(1_2+\begin{pmatrix}
\pIdeal{}^{n}&\pIdeal{}^{n+m}\\\pIdeal{}^{n+l}&\pIdeal{}^{n}\end{pmatrix}\bigr)
\!\Big/\!\bigl(1_2+\begin{pmatrix}
\pIdeal{}^{2n}&\pIdeal{}^{2n+m}\\\pIdeal{}^{2n+l}&\pIdeal{}^{2n}\end{pmatrix}\bigr)\!\simeq\!
\begin{pmatrix}\pIdeal{}^{n}/\pIdeal{}^{2n}
&\pIdeal{}^{n+m}/\pIdeal{}^{2n+m}\\
\pIdeal{}^{n+ l}/\pIdeal{}^{2n+l}&\pIdeal{}^{n}/\pIdeal{}^{2n}.
\end{pmatrix}
\end{equation*}
\end{lemma}
\begin{proof}First note that so long as the constraints of $l$, $m$, and $n$ are respected, the set of determinant one matrices of the form
\begin{equation*}
\left(1_2+\begin{pmatrix}
\pIdeal{}^{n}&\pIdeal{}^{n+m}\\\pIdeal{}^{n+l}&\pIdeal{}^{n}\end{pmatrix}\right)
\end{equation*} 
is a group.  Namely it is the intersection of $\grp{K}_n$, $\grp{B}_{n+l}$ and the stabilizer of $\left[\begin{smallmatrix}0\\1\end{smallmatrix}\right]_{n+m}$.\\
Again I calculate:
\begin{align*}
(1_2&+\pElem^n\begin{pmatrix}a_1&\pElem{}^{m}a_2\\
\pElem{}^{l}a_3&a_4
\end{pmatrix})(1_2+\pElem^n\begin{pmatrix}b_1&\pElem{}^{m}b_3\\
\pElem{}^{l}b_2&b_4
\end{pmatrix})\\
&=1_2+\pElem^n(\begin{pmatrix}a_1&\pElem{}^{m}a_2\\
\pElem{}^{l}a_3&a_4
\end{pmatrix}+\begin{pmatrix}b_1&\pElem{}^{m}b_3\\
\pElem{}^{l}b_2&b_4
\end{pmatrix})
+\pElem^{2n}\begin{pmatrix}
x_1&
\pElem{}^mx_2\\
\pElem{}^lx_3&
x_4
\end{pmatrix}\\
&=(1_2+\pElem^n\begin{pmatrix}a_1+b_1&\pElem{}^{m}(a_2+b_3)\\
\pElem{}^{l}(a_3+b_2)&d-(a_1+b_1)
\end{pmatrix})(1_2+\pElem^{2n}\begin{pmatrix}
x_1&
\pElem{}^mx_2\\
\pElem{}^lx_3&
x_4
\end{pmatrix}+r)\\
&\equiv(1_2+\pElem^n\begin{pmatrix}a_1+b_1&\pElem{}^{m}(a_2+b_3)\\
\pElem{}^{l}(a_3+b_2)&d-(a_1+b_1)
\end{pmatrix}).
\end{align*}
Where $x_i$ are the appropriate linear combinations of the $a_j\cdot b_k$, $d$ is the difference between $a_4+b_4$ and $-a_1-b_1$ and $r$ is a remainder in $\begin{pmatrix}\pIdeal{}^{2n}
&\pIdeal{}^{2n+m}\\
\pIdeal{}^{2n+l}&\pIdeal{}^{2n}.
\end{pmatrix}$.\\
I must verify that $d\in\pIdeal{}^{2n}$. To complete the proof I observe that
\begin{equation*}
(1+\pElem^{n}a_1)(1+\pElem^{n}a_4)-\pElem^{2n+m+l}a_2a_3=1,
\end{equation*}
so
\begin{equation*}
a_4+a_1=\pElem^{2n}(\pElem^{m+l}a_2a_3-a_1a_4).
\end{equation*}
An identical calculation for $b_4+b_1$ places $d$
clearly within $\pIdeal{}^{2n}$.
\end{proof}
I now can prove a slightly more general theorem than Theorem \ref{thm:dual}
\begin{namedtheorem}[$\mathbf{2^\prime}$]
If $n\geq 1$ and $-1\leq l,m\leq \ramInd{}$, then the set of characters of the quotient of determinant-one matrices
\begin{equation*}
\left(1_2+\begin{pmatrix}
\pIdeal{}^{n}&\pIdeal{}^{n+m}\\\pIdeal{}^{n+l}&\pIdeal{}^{n}\end{pmatrix}\right)
\Big/\left(1_2+\begin{pmatrix}
\pIdeal{}^{2n}&\pIdeal{}^{2n+m}\\\pIdeal{}^{2n+l}&\pIdeal{}^{2n}\end{pmatrix}\right)
\end{equation*}
as a $\grp{G}$-set is parametrized by trace-zero matrices of the form
\begin{equation*}
\begin{pmatrix}\pIdeal{}^{-2n-\ramInd{}}/\pIdeal{}^{-n-\ramInd{}}
&\pIdeal{}^{-2n-l}/\pIdeal{}^{-n-l}\\
\pIdeal{}^{-2n-m}/\pIdeal{}^{-n-m}&\pIdeal{}^{-2n-\ramInd{}}/\pIdeal{}^{-n-\ramInd{}}
\end{pmatrix}\end{equation*}
by the map
\begin{equation*}
\Mat{A}\mapsto\quotChar{\Mat{A}}.
\end{equation*}
\end{namedtheorem}
\begin{proof}
Consider the function $\quotChar{\Mat{A}}$ factored into three parts:
\begin{equation*}
\quotChar{\Mat{A}}= X\mapsto (X-1)\mapsto \trace{(X-1)\Mat{A}}\mapsto \chi\bigl(\trace{(X-1)\Mat{A}}\bigr).
\end{equation*}
As each factor of $\quotChar{\Mat{A}}$ is a group homomorphism (Lemmas \ref{lem:pair}, \ref{lem:Gmap} and \ref{lem:Xminus1}), $\quotChar{\Mat{A}}$ is indeed a character. \\
By the linearity of the pairing $\BLPair{\Mat{B}}{\Mat{A}}$ (Lemma \ref{lem:pair}), if $\quotChar{\Mat{A}}=\quotChar{\Mat{A}^\prime}$ then \linebreak $\quotChar{\Mat{A}}/\quotChar{\Mat{A}^\prime}=\quotChar{\Mat{A}-\Mat{A}^\prime}\equiv 1$ (Lemma \ref{lem:Gmap}).  In such a case $\Mat{A}-\Mat{A}^\prime\in \begin{pmatrix}\pIdeal{}^{-n-\ramInd{}}
&\pIdeal{}^{-n-l}\\
\pIdeal{}^{-n-m}&\pIdeal{}^{-n-\ramInd{}}
\end{pmatrix}$, by the non-degeracy of the pairing $\BLPair{\Mat{B}}{\Mat{A}}$ (Lemma \ref{lem:pair}). Hence, the map is 1-to-1.\\
Since the factor $X\mapsto X-1$ induces an isomorphism to an abelian groups (Lemma \ref{lem:Xminus1}), Pontryagin duality implies that the characters of the multiplicative quotient are in one-to-one correspondence with the elements of the quotient itself. \\
The map is onto since:
\begin{equation*}
\left|\begin{pmatrix}\pIdeal{}^{n}/\pIdeal{}^{2n}
&\pIdeal{}^{n+m}/\pIdeal{}^{2n+m}\\
\pIdeal{}^{n+ l}/\pIdeal{}^{2n+l}
&\pIdeal{}^{n}/\pIdeal{}^{2n}
\end{pmatrix}\right|
=
\left|\begin{pmatrix}\pIdeal{}^{-2n-\ramInd{}}/\pIdeal{}^{-n-\ramInd{}}
&\pIdeal{}^{-2n-l}/\pIdeal{}^{-n-l}\\
\pIdeal{}^{-2n-m}/\pIdeal{}^{-n-m}
&\pIdeal{}^{-2n-\ramInd{}}/\pIdeal{}^{-n-\ramInd{}}
\end{pmatrix}
\right|.
\end{equation*}

\end{proof}
\bibliographystyle{plain}
{\small\bibliography{M-PMaps}}

\begin{thebibliography}{10}

\bibitem{C1972}
W.~Casselman.
\newblock On the representations of {$SL_2(k)$} related to binary quadratic
  forms.
\newblock {\em American Journal of Mathematics}, 94(3):810--834, Jul. 1972.

\bibitem{ADSS2011}
Paul J. Sally~Jr. Jefferey D.~Adler, Stephen~DeBacker and Loren Spice.
\newblock Supercuspidal characters of {$\mathrm{SL}_2$} over a {$p$}-adic
  field.
\newblock In Loren~Spice Robert S.~Doran, Paul J. Sally~Jr., editor, {\em
  Harmonic Analysis on Reductive, p-adic Groups}, pages 19--70. Amer. Math.
  Soc., 2011.
\newblock AMS special session, harmonic analysis and representations of
  reductive, $p$-adic groups, January 16, 2010.

\bibitem{K1978p1}
P.C. Kutzko.
\newblock On the supercuspidal representations of {$Gl_2$}.
\newblock {\em American Journal of Mathematics}, 100(1):43--60, February 1978.

\bibitem{K1978p2}
P.C. Kutzko.
\newblock On the supercuspidal representations of {$Gl_2$}, {II}.
\newblock {\em American Journal of Mathematics}, 100(4):705--716, August 1978.

\bibitem{K1972}
Phillip~Caesar Kutzko.
\newblock {\em The Characters of the Binary Modular Congruence Group}.
\newblock PhD thesis, University of Wisconsin, January 1972.

\bibitem{M1}
D.~Manderscheid.
\newblock On the supercuspidal representations of {$SL_2$} and its two-fold
  cover. {I}.
\newblock {\em Math. Ann.}, 266:287--295, 1984.

\bibitem{M2}
D.~Manderscheid.
\newblock On the supercuspidal representations of {$SL_2$} and its two-fold
  cover. {II}.
\newblock {\em Math. Ann.}, 266:297--305, 1984.

\bibitem{MP1994}
Allen Moy and Gopal Prasad.
\newblock Unrefined minimal $k$-types for $p$-adic groups.
\newblock {\em Invent. Math.}, 116(no. 1–3):393–408, 1994.

\bibitem{NW1976}
Alexander Nobs and Jurgen Wolfart.
\newblock Die irreduziblen darstellungen der gruppen {$\SL{2}{Z_p}$} ,
  insbesondere {$\SL{2}{Z_2}$} {II}.
\newblock {\em Comment. Math. Helv.}, 51:491–526, 1976.

\bibitem{S1977}
Jean-Pierre Serre.
\newblock {\em Trees}.
\newblock Springer Monographs in Mathematics. Springer-Verlag, first english
  edition, 2003, 1980.
\newblock Translation of {\textit{Arbres, Amalgames,{$\mathrm{SL}_{2}$}.
  Asterisque no. 46}}, from the French by John Stillwell.

\bibitem{Sh1966}
Joseph~A. Shalika.
\newblock {\em Representation of the Two by Two Unimodular Group Over Local
  Fields}.
\newblock PhD thesis, Johns Hopkins University, 1966.

\bibitem{K-S2011}
Terence Joseph~K\i vran Swaine.
\newblock {\em The Admissible Dual of SL(2) of the Dyadic Numbers}.
\newblock PhD thesis, The City University of New York, May 2011.

\bibitem{W1964}
Andr\'{e} Weil.
\newblock Sur certains groupes d$'$op\'{e}rateurs unitaires.
\newblock {\em Acta. Math.}, 111:143–21, 1964.

\bibitem{Y2001}
Jiu-Kang Yu.
\newblock Construction of tame supercuspidal representations.
\newblock {\em J. Amer. Math. Soc.}, 14(3):579--622, 2001.
\newblock electronic.

\end{thebibliography}
\end{document}